\documentclass[12pt,reqno,a4paper]{amsart}
\usepackage{amssymb,amscd,epsf,textcomp}
\usepackage{hyperref}  

\begin{document}

\let\kappa=\varkappa
\let\epsilon=\varepsilon
\let\phi=\varphi
\let\p\partial
\let\lle=\preccurlyeq
\let\ulle=\curlyeqprec

\def\Z{\mathbb Z}
\def\R{\mathbb R}
\def\C{\mathbb C}
\def\Q{\mathbb Q}
\def\P{\mathbb P}
\def\HH{\mathsf{H}}
\def\XX{\mathcal X}

\def\conj{\overline}
\def\Beta{\mathrm{B}}
\def\const{\mathrm{const}}
\def\ov{\overline}
\def\wt{\widetilde}
\def\wh{\widehat}

\renewcommand{\Im}{\mathop{\mathrm{Im}}\nolimits}
\renewcommand{\Re}{\mathop{\mathrm{Re}}\nolimits}
\newcommand{\codim}{\mathop{\mathrm{codim}}\nolimits}
\newcommand{\Aut}{\mathop{\mathrm{Aut}}\nolimits}
\newcommand{\lk}{\mathop{\mathrm{lk}}\nolimits}
\newcommand{\sign}{\mathop{\mathrm{sign}}\nolimits}
\newcommand{\rk}{\mathop{\mathrm{rk}}\nolimits}

\def\id{\mathrm{id}}
\def\Leg{\mathrm{Leg}}
\def\Jet{{\mathcal J}}
\def\sS{{\mathcal S}}
\def\lcan{\lambda_{\mathrm{can}}}
\def\ocan{\omega_{\mathrm{can}}}

\renewcommand{\mod}{\mathrel{\mathrm{mod}}}

\newtheorem{mainthm}{Theorem}
\renewcommand{\themainthm}{{\Alph{mainthm}}}
\newtheorem{thm}{Theorem}[section]
\newtheorem{lem}[thm]{Lemma}
\newtheorem{prop}[thm]{Proposition}
\newtheorem{cor}[thm]{Corollary}

\theoremstyle{definition}
\newtheorem{exm}[thm]{Example}
\newtheorem{rem}[thm]{Remark}
\newtheorem{df}[thm]{Definition}

\numberwithin{equation}{section}

\title{Universal orderability of Legendrian isotopy classes}
\author[Chernov \& Nemirovski]{Vladimir Chernov and Stefan Nemirovski}
\thanks{This work was partially supported by a grant from the Simons Foundation (\#{}235674 to Vladimir Chernov).
The second author was partially supported by SFB/TR~12 of the DFG and RFBR grant \textnumero 13-01-12417-ofi-m}
\address{Department of Mathematics, 6188 Kemeny Hall,
Dartmouth College, Hanover, NH 03755-3551, USA}
\email{Vladimir.Chernov@dartmouth.edu}
\address{%
Steklov Mathematical Institute, Gubkina 8, 119991 Moscow, Russia;\hfill\break
\strut\hspace{8 true pt} Mathematisches Institut, Ruhr-Universit\"at Bochum, 44780 Bochum, Germany}
\email{stefan@mi.ras.ru}

\begin{abstract}
It is shown that non-negative Legendrian isotopy defines a partial
order on the universal cover of the Legendrian isotopy class of
the fibre of the spherical cotangent bundle of any manifold.
This result is applied to Lorentz geometry in the spirit of
the authors' earlier work on the Legendrian Low conjecture.
\end{abstract}

\maketitle

\section{Introduction}

\subsection{Partial orders in contact geometry}
Let $(X,\ker\alpha)$ be a contact manifold with a co-oriented contact structure.
A contact or Legendrian isotopy in~$X$ is called non-negative if individual points
move in the direction of the co-orientation of the contact hyperplanes
(formal definitions are given in~\S\ref{NonNegLeg&Cont}).

Let $\mathcal C$ denote either a connected component of the contactomorphism group of $X$
or a Legendrian isotopy class in~$X$.
We write $a\lle b$ for two elements $a,b\in\mathcal C$ if there is a non-negative isotopy
connecting $a$ to~$b$. This partial relation has a natural lift
to a partial relation $\ulle$ on the universal cover~$\widetilde{\mathcal C}$,
see \S\ref{PartOrders} for details.
It is clear that $\lle$ and $\ulle$ are reflexive and transitive.
Let us call $\mathcal C$ {\it orderable\/} if $\lle$ is also
antisymmetric (i.e.\ defines a partial order on $\mathcal C$) and {\it universally
orderable\/} if $\ulle$ defines a partial order on~$\widetilde{\mathcal C}$.

The question of (universal) orderability for groups of contactomorphisms
and Legendrian isotopy classes was apparently first raised by
Eliashberg and Polterovich~\cite{ElPo} and Bhupal~\cite{Bh}.
There are now several papers treating various situations
\cite{ElKiPo}, \cite{CFP}, \cite{ChNe1}, \cite{ChNe2}, \cite{Sa}, \cite{AlFr}, \cite{Za}, \cite{AlMe}, \cite{BoZa}.
For reasons that will be explained in~\S\ref{c&o},
we are particularly interested in the case,
first considered by Colin, Ferrand and Pushkar' in~\cite{CFP}, when $\mathcal C=\Leg(ST^*_xM)$
is the Legendrian isotopy class of the fibre of the spherical cotangent bundle $ST^*M$
of a manifold~$M$, $\dim M\ge 2$.
This class is orderable if the universal cover of $M$ is non-compact by~\cite[Remark~8.2]{ChNe2}
or does not have the integral cohomology ring of a compact rank one symmetric space (CROSS)
by~\cite[Theorem~1.13]{FrLaSch} and Proposition~\ref{Order}.
On the other hand, $\Leg(ST^*_xM)$ is not orderable for every CROSS and,
more generally, for any manifold $M$ admitting a Riemannian $Y^x_\ell$-metric,
see~\cite[Example~8.3]{ChNe2}. The following special case of Theorem~\ref{UnivOrder}
shows that {\it universal\/} orderability holds for every~$M$.

\begin{thm}
\label{uorder}
The Legendrian isotopy class of the fibre of $ST^*M$ is
universally orderable.
\end{thm}

This theorem is inspired by and generalises the result
of Eliashberg, Kim, and Polterovich \cite[Theorem 1.18]{ElKiPo}
about the contactomorphism group of $ST^*M$.
They proved (modulo an assumption removed in~\cite{ChNe2})
that the identity component $\mathrm{Cont}_0(ST^*M)$
is {\it universally\/} orderable for every closed manifold~$M$.
However, $\mathrm{Cont}_0(ST^*M)$ is not orderable
for  any manifold~$M$
admitting a Riemannian metric with periodic geodesic flow.

\subsection{Causality and orderability}
\label{c&o}
Let $(\XX,g)$ be a spacetime, that is, a time-oriented
connected Lorentz manifold. Assume that $g$ has
signature $(+,-,\dots,-)$ with at least two negative spacelike
directions. A piecewise smooth curve in $\XX$ is called future-directed
(f.d.) if its tangent vector at each point lies in the future
hemicone defined by the time-orientation in the non-spacelike cone
of the Lorentz metric.

The {\it causality relation\/} $\le$  on $\XX$ is defined
by setting $x\le y$ if either $x=y$ or there is a f.d.\ curve
connecting $x$ to~$y$. This relation is always reflexive
and transitive. If it is a partial order, the spacetime is
said to be {\it causal}.

A causal spacetime is {\it globally hyperbolic\/} if all
causal segments $I_{x,y}=\{z\in\XX\mid x\le z\le y\}$
are compact~\cite{BeSa2}. By another result of Bernal and S\'anchez~\cite{BeSa1},
a spacetime is globally hyperbolic if and only if it
contains a smooth spacelike hypersurface $M\subset\XX$
such that every endless f.d.\ curve meets $M$ exactly once.
Such an $M$ is called a {\it Cauchy surface\/} in~$\XX$.

Let $\mathfrak N$ be the set of all f.d.\ non-parametrised null
geodesics (i.e.\ light rays) in $(\XX,g)$.
$\mathfrak N$~has a canonical structure of a contact manifold,
see~\cite[\S 2]{Lo2} or~\cite[pp.\,252--253]{NaTo}.
The set of all null geodesics passing through
a point $x\in\XX$ is a Legendrian sphere $\mathfrak S_x\subset\mathfrak N$
called the {\it sky\/} (or the {\it celestial sphere\/}) of that point.
The association $x\mapsto {\mathfrak S}_x$ was one of the starting
points of Penrose's twistor theory, see e.g.~\cite{Pe}; its study in the context of
contact geometry was initiated by Low~\cite{Lo1}.

There is a contactomorphism
$$
\rho_M:\mathfrak N\overset{\cong}{\longrightarrow} ST^*M
$$
taking a null geodesic $\gamma\in\mathfrak N$ to the equivalence class of
the $1$-form $g(\dot\gamma,\cdot)$ on $T_{\gamma\cap M}M$,
where $\dot\gamma$ is a f.d.\ tangent vector to $\gamma$ at~$\gamma\cap M$.
For every sky, its image $\rho_M({\mathfrak S}_x)$ in $ST^*M$
is Legendrian isotopic to the fibre of $ST^*M$. Hence, we obtain a map
$$
\mathfrak s:\XX \longrightarrow \mathrm{Leg}(ST_{\{\mathrm{pt}\}}^*M)
$$
from a globally hyperbolic spacetime $\XX$ to the Legendrian isotopy class
of the fibre of the spherical cotangent bundle of its Cauchy surface~$M$.

The following key observation is an immediate corollary of the proof
of~\cite[Proposition~4.2]{ChNe1} taking into account the
opposite convention for the signature of the Lorentz metric.

\begin{prop}
\label{c2o}
$x\le y \Longrightarrow \mathfrak s(x)\lle\mathfrak s(y)$.
\end{prop}

The converse implication does not hold in general.
(For example, the map $\mathfrak s$ may be non-injective,
cf.~\cite[Example~10.5]{ChNe2}.)
However, a useful sufficient condition for it to hold
can be formulated purely in terms of~$ST^*M$.
This condition was an implicit underpinning
of our work on the (Legendrian) Low conjecture in~\cite{ChNe1,ChNe2}.

\begin{prop}
\label{o2c}
If\/ $\mathrm{Leg}(ST_{\{\mathrm{pt}\}}^*M)$ is orderable,
then $\mathfrak s(x)\lle\mathfrak s(y)$ $\Longrightarrow$ $x\le y$.
\end{prop}

\begin{proof}
If $y\le x$, then $\mathfrak s(y)\lle\mathfrak s(x)$
by Proposition~\ref{c2o} and hence $\mathfrak s(x)=\mathfrak s(y)$
by orderability. If $x$ and $y$ are not causally related,
then the Legendrian links $\mathfrak s(x)\sqcup \mathfrak s(y)$
and $\mathfrak s(y)\sqcup \mathfrak s(x)$ are Legendrian
isotopic (through links formed by skies of pairs of causally unrelated points)
by \cite[Lemma 4.3]{ChNe1}.
By the Legendrian isotopy extension theorem,
there exists a $\phi\in\mathrm{Cont}_0(ST^*M)$ such that
$\phi(\mathfrak s(x)\sqcup \mathfrak s(y))=\mathfrak s(y)\sqcup \mathfrak s(x)$.
It follows that $\mathfrak s(y)\lle\mathfrak s(x)$ and
again $\mathfrak s(x)=\mathfrak s(y)$ by orderability.

So we have to exclude the possibility that $x\ne y$
but $\mathfrak s(x)=\mathfrak s(y)$. In this case,
$x$ and $y$ are causally related by null geodesics.
Assume that $y\le x$ (otherwise we are done).
Moving from $y$ a bit along any f.d.\ null geodesic, we
obtain a point $z$ with a different sky and such
that $y\le z\le x$. Then $\mathfrak s(y)\lle\mathfrak s(z)\lle\mathfrak s(x)$
by Proposition~\ref{c2o} and hence $\mathfrak s(z)=\mathfrak s(x)=\mathfrak s(y)$
by orderability, which contradicts the choice of~$z$.
\end{proof}

It follows now from the orderability results cited above that the conclusion
of Proposition~\ref{o2c} holds for a globally hyperbolic spacetime such that
the universal cover of its Cauchy surface is either non-compact or does not
have the integral cohomology ring of a CROSS.
The remaining cases (e.g.~the case when $M$ is homotopy equivalent to a sphere)
may be handled using the universal orderability result of the present paper.

Firstly, note that we may pass to a simply connected globally hyperbolic
spacetime by considering the universal cover of $\XX$ with the pulled back
Lorentz metric, see~\cite[Theorem~14]{ChRu}. If $\XX$ is simply connected,
the map $\mathfrak s$ admits a lift
$$
\widetilde{\mathfrak s}:\XX\longrightarrow \widetilde{\mathrm{Leg}}(ST_{\{\mathrm{pt}\}}^*M).
$$
A careful inspection of the proofs of Propositions~\ref{c2o} and~\ref{o2c}
shows that they remain true with $\Leg$, ${\mathfrak s}$, and $\lle$
replaced by $\widetilde{\Leg}$, $\widetilde{\mathfrak s}$, and $\ulle$.
Since $\ulle$ is always a partial order by Theorem~\ref{uorder},
we obtain the following result.

\begin{thm}
Suppose that $\XX$ is a simply connected globally hyperbolic spacetime
with Cauchy surface~$M$. Then
$$
x\le y\mbox{ in }\XX\quad\Longleftrightarrow\quad
\widetilde{\mathfrak s}(x)\ulle\widetilde{\mathfrak s}(y)
\mbox{ in\/ } \widetilde{\mathrm{Leg}}(ST_{\{\mathrm{pt}\}}^*M)
$$
\end{thm}

Thus, up to passing to a finite cover, the causality relation
(and hence, by~\cite[Theorem~2]{Ma}, the conformal Lorentz structure)
of a globally hyperbolic spacetime
is always determined by the map $x\mapsto\mathfrak S_x$
to the space of Legendrian spheres in its space of null geodesics.

\subsection*{Acknowledgments}
This paper owes very much to the seminal work of Eliashberg and Polterovich~\cite{ElPo}.
To a large extent, it implements their original strategy for proving \cite[Theorem~1.18]{ElKiPo},
cf.~Remark~\ref{shortcut}. The authors are also very grateful to Yuli Rudyak
for his valuable advice on the proof of Theorem~\ref{NonDispl}.

\section{A Legendrian non-displacement result}

\subsection{Generating hypersurfaces for Legendrian submanifolds in
spherical cotangent bundles (after Eliashberg and Gromov~\cite[\S 4.2]{ElGr})}
Let $L\subset ST^*M$ be a Legendrian submanifold in the spherical cotangent
bundle of a closed manifold~$M$.
Suppose that there exists a function $f:M\times\R^N\to\R$, $N\ge 0$, such that
\begin{itemize}
\item[1)] $0$ is not a critical value of $f$;
\item[2)] the hypersurface $\{f=0\}$ is in general position with respect to
the projection $\pi_M:M\times\R^N\to M$, that is to say, the subset
$$
FT_f:=\{x\in M\times\R^N\mid \{f=0\} \mbox{ is tangent to } \{\pi_M(x)\}\times\R^N\}
$$
is a submanifold cut out transversally by the equations
$f=0$ and ${df|}_{\R^N}=0$;
\item[3)] the map
$$
FT_f \ni x \longmapsto [\pi_{M*} df(x)] \in ST_{\pi_M(x)}^*M,
$$
where $\pi_{M*}df(x)$ is the unique $1$-form at $\pi_M(x)$ such that
$df(x)$ is its pull-back by~$\pi_M$,
defines a diffeomorphism $FT_f\stackrel{\cong}{\longrightarrow} L$;
\item[4)] $f$ is equal to a non-degenerate quadratic form $Q:\R^N\to\R$ outside of
a compact subset in $M\times\R^N$.
\end{itemize}
Then $\{f=0\}$ is called a {\it quadratic at infinity generating hypersurface\/}
for the Legendrian submanifold~$L\subset ST^*M$.

\begin{exm}
\label{ExampleGen}
Let $f:M\to\R$ be a smooth function such that $0$ is not a critical value.
Then the hypersurface $\{f=0\}\subset M\; (=M\times\R^0)$ generates the
Legendrian submanifold
$$
L_f=\{[df(x)]\in ST^*M \mid f(x)=0\}.
$$
Thus, in this case $L$ is the Legendrian lift of the co-oriented generating hypersurface
and this hypersurface is the wave front of $L$.
\end{exm}

If $Q':\R^{N'}\to\R$ is another non-degenerate quadratic form
and $\phi:\R^{N'}\to [0,1]$ is a cut-off function with $\|d\phi\|\le 1$
that is equal to $1$ on a sufficiently large ball, then the zero set
of the function
$$
\phi\cdot(f-Q) + Q + Q':M\times\R^{N+N'}\longrightarrow\R
$$
is also a quadratic at infinity generating hypersurface for the same Legendrian
submanifold~$L$. This operation on generating hypersurfaces is called {\it stabilisation}.

\begin{thm}[{cf.~\cite[Theorem 4.2.1]{ElGr}}]
\label{GenHyp}
Suppose that $\{L_t\}_{t\in [0,1]}$ is a Legendrian isotopy in $ST^*M$
such that $L_0=L_f$ for a function $f:M\to\R$.
Then there exist an $N\ge 0$ and a smooth family of quadratic at infinity generating
hypersurfaces $\{f_t=0\}\subset M\times\R^N$ for $L_t$ such that $\{f_0=0\}$ is
a stabilisation of $\{f=0\}$.
\end{thm}

\begin{rem}
An inaccuracy in the proof of~\cite[Theorem 4.2.1]{ElGr} was
pointed out and corrected by Pushkar' in two recent preprints~\cite{Pu1} and~\cite{Pu2}.
A result similar to Theorem~\ref{GenHyp} may also be found in~\cite{Fe}.
In a somewhat different context, generating hypersurfaces appeared
in the classical text~\cite[\S 20.7]{AVG}.
\end{rem}

\begin{exm}
\label{Conormal}
The spherical conormal bundle $SN^*V\subset ST^*M$ of a submanifold $V\subset M$
is a Legendrian submanifold. For instance, if $V$ is a point $v\in M$, then
$SN^*V=ST^*_vM$ is the fibre of $ST^*M$ at~$v$. The co-geodesic flow of a Riemann metric
on $M$ defines a Legendrian isotopy of $SN^*V$ to the Legendrian lift of
the co-oriented boundary of a geodesic tube around~$V$. Taking a function
$f:M\to\R$ such that $\{f<0\}$ is such a tube, we see that $SN^*V$
is Legendrian isotopic to a submanifold of the form considered in Example~\ref{ExampleGen}.
In particular, it has a quadratic at infinity generating hypersurface.
\end{exm}

\subsection{Legendrian non-displacement in the spherical cotangent bundle
of a fibred manifold}
Let $\pi: P\to M$ be a submersion. For every subset $U\subset ST^*M$, its pull-back
in $ST^*P$ is the subset
$$
\pi^*(U):=\left\{ [\pi^*\xi]\in ST^*P \mid [\xi]\in U\right\}.
$$
The pull-back of a Legendrian submanifold in $ST^*M$
is a Legendrian submanifold of $ST^*P$. For instance,
if $f:M\to\R$ is a function such that $0$ is not a critical value
and $L_f$ is the Legendrian submanifold in $ST^*M$ generated by $\{f=0\}$
as in Example~\ref{ExampleGen},
then
$$
\pi^*(L_f) = L_{f\circ\pi} \subset ST^*P.
$$
Similarly, the pull-back of the spherical conormal bundle of a submanifold $V\subset M$
is the spherical conormal bundle of its pre-image~$\pi^{-1}(V)\subset P$.

\begin{thm}
\label{NonDispl}
Let $P$ and $M$ be closed connected manifolds and $\pi:P\to M$ a fibre bundle.
Any Legendrian submanifold of $ST^*P$ Legendrian isotopic to $\pi^*(L_f)$
for some function $f:M\to\R$ intersects~$\pi^*(ST^*M)$.
\end{thm}

\begin{proof}
Let us assume that there is a Legendrian submanifold $L\subset ST^*P-\pi^*(ST^*M)$ that
is Legendrian isotopic to $\pi^*(L_f)$.
The Legendrian submanifold $\pi^*(L_f)=L_{f\circ\pi}$ has a (trivially) quadratic
at infinity generating hypersurface $\{f\circ\pi=0\}\subset P$. By Theorem~\ref{GenHyp},
there exists a family of hypersurfaces $H_t=\{f_t=0\}\subset P\times\R^N$ such that
\begin{itemize}
\item[1)] $H_0=\{f_0=0\}$ is a stabilisation of $\{f\circ\pi=0\}\subset P$;
\item[2)] $H_1=\{f_1=0\}$ generates~$L$;
\item[3)] outside of a compact subset of $P\times\R^N$, all hypersurfaces $H_t$ coincide
with the hypersurface $P\times\{Q=0\}$, where $Q$ is a non-degenerate quadratic
form on~$\R^N$.
\end{itemize}
Applying additional stabilisation, if necessary, we may assume that both
inertia indices $\varkappa_\pm(Q)\ge 2$. This guarantees, in particular,
that the hypersurfaces $H_t$ are connected.

The assumption that $L$ is disjoint from $\pi^*(ST^*M)$ means
that its generating hypersurface $H_1$
is nowhere tangent to the fibres of the composite projection
$\pi_M:P\times\R^N\to P\stackrel{\pi}{\longrightarrow} M$.
In other words, the restriction ${\pi_M|}_{H_1}$ is a submersion.
Since $H_1$ is standard at infinity, it follows by Ehresmann's theorem
that ${\pi_M|}_{H_1}: H_1\to M$ is a fibre bundle.

On the other hand, the fibre of the projection ${\pi_M|}_{H_0}: H_0\to M$ over a point $x\in M$
is essentially a product of the form $\pi^{-1}(x)\times \{Q=-f(x)\}$.
In particular, the fibres over the sets $\{f>0\}$ and $\{f<0\}$ are embedded
in topologically different ways. (This statement will be made precise at
the end of the proof.)

The main point of the following somewhat technical argument is to show that this behaviour
of the fibres contradicts the fact that the hypersurfaces $H_0$ and $H_1$ are isotopic within the
class of hypersurfaces satisfying condition~(3).

Let us fix a parametrisation
$$
\iota_t:H\longrightarrow H_t, \qquad t\in [0,1],
$$
of the family $\{H_t\}$ such that $\iota_t$ is the same standard
embedding outside of a compact subset in~$H$. Let further
$$
\pi_t := {\pi_M|}_{H_t}\circ \iota_t: H\longrightarrow M
$$
be the induced family of projections to~$M$.

Since $\pi_1:H\to M$ is a fibre bundle, we may apply
the relative homotopy lifting property~\cite[Proposition~4.48]{Ha}
to the homotopy of maps
$$
\pi_t: H\longrightarrow M, \qquad t\in[0,1].
$$
It follows that there is a homotopy of (continuous) maps
$$
\phi_t :H\longrightarrow H, \qquad t\in[0,1],
$$
such that $\phi_1\equiv\id_H$, $\phi_t=\id_H$ outside of a compact set for all~$t$,
and
$$
\pi_t = \pi_1\circ \phi_t, \qquad t\in [0,1].
$$
Thus, $\pi_0=\pi_1\circ\phi_0$, where $\phi_0:H\to H$ is homotopic to
the identity in the class of compactly supported self-maps of~$H$.

Suppose now that $\beta\in \HH_c^*(H;\Z/2)$ is a cohomology class with
compact support such that its restriction to a fibre $\pi_0^{-1}(x)$ is
non-zero. We claim that its restriction to every regular fibre of $\pi_0$
must also be non-zero.
Indeed, note first that
$$
0\ne \beta|_{\pi_0^{-1}(x)} = (\phi_0^*\beta)|_{\pi_0^{-1}(x)} = (\phi_0|_{\pi_0^{-1}(x)})^*(\beta|_{\pi_1^{-1}(x)}),
$$
where the first equality holds because $\phi_0$ acts as the identity on the cohomology of $H$
and the second because $\phi_0$ maps the fibre of $\pi_0$ over~$x$ to the fibre of $\pi_1$
over~$x$.
Since the fibres of the fibre bundle $\pi_1$ are all isotopic,
we conclude that
$$
0\ne \beta|_{\pi_1^{-1}(y)}\quad\mbox{ for all } y\in M.
$$
However, if $y$ is a regular value of $\pi_0$, then the restriction
$$
\phi_0|_{\pi_0^{-1}(y)}: \pi_0^{-1}(y)\longrightarrow \pi_1^{-1}(y)
$$
is a proper map of equidimensional manifolds that has $\Z/2$-degree~$1$.
It follows from Poincar\'e duality and naturality of the $\cup$-product that
this map induces an injection on $\Z/2$-cohomology with compact support
(cf.~\cite[Lemma~2.2]{Wa}).
Hence,
$$
\beta|_{\pi_0^{-1}(y)} = (\phi_0^*\beta)|_{\pi_0^{-1}(y)} = (\phi_0|_{\pi_0^{-1}(y)})^*(\beta|_{\pi_1^{-1}(y)})\ne 0,
$$
as claimed.

Thus, to complete the proof of the theorem by contradiction, we need to exhibit
a compactly supported cohomology class on $H$ such that its restrictions to the
regular fibres of $\pi_0$ can be both non-zero and zero. Let $V_+\subset\R^N$
be a maximal linear subspace on which the form $Q$ is positive definite.
By construction, $\dim V_+=\varkappa_+\ge 2$.
Consider the submanifold
$P\times V_+\subset P\times \R^N$. Its intersection with $H_0$ is compact
and defines (by duality) a compactly supported cohomology class on $H_0$.
Let $\beta$ be the pull-back of this class to~$H$.

The restriction of $\beta$ to a regular fibre $\pi_0^{-1}(x)$ is then dual to the
pre-image of the intersection of the fibre of ${\pi_M|}_{H_0}$ over $x$ with $P\times V_+$.
If $f(x)>0$, this intersection is empty. If $f(x)<0$, it is the product
$\pi^{-1}(x)\times S$, where the sphere $S=V_+\cap \{Q=-f(x)\}$ is the `waist'
of the quadric, and so the dual compactly supported cohomology class
on $\pi^{-1}(x)\times \{Q=-f(x)\}$ is non-zero. Hence, $\beta$ has the required property.
\end{proof}

\begin{rem}
To explain the idea of the proof, let us apply it directly to the simplest case when $N=0$
and so $H_t$ are (let us say, connected) hypersurfaces in~$P$ isotopic to $H_0=\pi^{-1}(\{f=0\})$.
Choosing a parametrisation $\iota_t:H\to H_t$, we get a homotopy of maps $\pi_t=\pi\circ\iota_t:H\to M$
such that $\pi_1$ is a fibre bundle projection and $\pi_0$ is not surjective
(because its image is $\{f=0\}\subsetneq M$). Using the homotopy
lifting property, we obtain a homotopy $\phi_t:H\to H$ such that $\phi_1=\id_H$
and $\pi_0=\pi_1\circ\phi_0$. It follows from the latter equality that $\phi_0$ can not be surjective.
This contradicts the fact that a map homotopic to the identity map of the {\it closed\/} manifold~$H$
has degree~$1$ and hence must be onto.
\end{rem}

\section{Transverse families of Legendrian submanifolds}

\subsection{Families}
A parametrised family of Legendrian submanifolds in a contact manifold $(X,\ker\alpha)$
over a base~$B$ is a map
$$
F: \underline{L}\longrightarrow X,
$$
where $\pi:\underline{L}\to B$ is a fibre bundle and the restriction
$$
F|_{\pi^{-1}(b)}: \pi^{-1}(b) \longrightarrow L_b\subset X
$$
is a Legendrian embedding for every~$b\in B$.
Two parametrised families $F_1,F_2: \underline{L}\to X$
are called equivalent if $F_1 = F_2\circ \Phi$
for a diffeomorphism $\Phi:\underline{L}\to \underline{L}$ such that $\pi\circ\Phi=\pi$.

\begin{df}
A family $\mathcal L=\{L_b\}_{b\in B}$ of Legendrian submanifolds
is an equivalence class of parametrised families.
\end{df}

A family of Legendrian submanifolds is called {\it constant\/} if $L_b$
is the same Legendrian submanifold in $X$ for all~$b\in B$.
Note that a constant family may have non-constant parametrisations.

\begin{df}
A family of Legendrian submanifolds in $(X,\ker\alpha)$ is called {\it transverse\/}
if it has a parametrisation $F: \underline{L}\to X$
such that the pull-back $F^*\alpha$ of the contact form
does not vanish anywhere on $\underline{L}$.
\end{df}

It is obvious that this property depends neither on the choice
of a parametrisation of the family nor on the choice of a
contact form defining the contact structure on $X$.

\begin{exm}[Transverse families and positive isotopies]
The simplest example of a family of Legendrian submanifolds is
a Legendrian isotopy $\{L_t\}_{t\in [0,1]}$ that can be parametrised
by the product $L_0\times [0,1]$. This family is transverse
if and only if the isotopy is either positive or negative,
see~\S\ref{NonNegLeg&Cont}.
\end{exm}

\begin{exm}[Transverse families and fibre bundles]
Let $p:M\to B$ be a fibre bundle. Then $\{SN^*p^{-1}(b)\}_{b\in B}$
is a transverse family of Legendrian submanifolds in $ST^*M$
with base $B$. It is tautologically parametrised
by $\underline{L}=p^*(ST^*B)$ and the projection
$\pi:\underline{L}\to B$ is the composition
$$
p^*(ST^*B)\subset ST^*M\longrightarrow M\stackrel{p}{\longrightarrow} B
$$
In particular, one can take $p=\id_M$ and obtain the family
of all fibres of $ST^*M$.
\end{exm}

\subsection{Stabilisation of contact manifolds}
\label{Stab}
Let $(X,\ker\alpha)$ be a contact manifold and $B$ an arbitrary manifold.
The {\it $B$-stabilisation\/} of $X$ is the contact manifold
$$
X^B := (X\times T^*B,\ker(\alpha\oplus\lcan)),
$$
where $\lcan = p\,dq$ is the canonical $1$-form on $T^*B$.
If $\alpha'=e^f\alpha$ is another contact form defining the same
(co-oriented) contact structure on $X$, then the map
\begin{equation}
\label{StabInvar}
X\times T^*B\ni (x,q,p) \longmapsto (x,q,e^{f(x)}p)\in X\times T^*B
\end{equation}
defines a contactomorphism
$$
(X\times T^*B,\ker(\alpha\oplus\lcan))\stackrel{\cong}{\longrightarrow}(X\times T^*B,\ker(\alpha'\oplus\lcan)).
$$
Hence, the $B$-stabilisation of $X$ is well-defined as a contact manifold.

\begin{exm}
\label{StabST*M}
Suppose that $X=ST^*M$ is the spherical cotangent bundle of a manifold $M$
with its canonical contact structure. Then the $B$-stabilisation of $X$ is
naturally contactomorphic to the open subset
$$
ST^*(M\times B) - \pi_B^*(ST^*B)
$$
of the spherical cotangent bundle of the product $M\times B$,
where $\pi_B:M\times B\to B$ is the projection.
Indeed, let $\alpha$ be any contact form on $ST^*M$.
There exists a unique fibrewise starshaped embedding
$\iota:ST^*M\hookrightarrow T^*M$ such that $\iota^*\lcan=\alpha$.
The map
$$
ST^*M\times T^*B \ni (\xi,\eta) \longmapsto [\iota(\xi)\oplus\eta] \in S(T^*M\times T^*B) = ST^*(M\times B)
$$
defines a contactomorphism
$$
(ST^*M\times T^*B, \ker(\alpha\oplus\lcan))\stackrel{\cong}{\longrightarrow} ST^*(M\times B) - \pi_B^*(ST^*B).
$$
\end{exm}

\subsection{Legendrian suspension}
\label{Susp}
To each Legendrian family $\mathcal L$ in $X$ with base $B$, we can associate a Legendrian
submanifold $\widetilde{\mathcal L}$ in the $B$-stabilisation of~$X$
called the {\it Legendrian suspension\/} of $\mathcal L$.
Indeed, let $F:\underline{L}\to X$ be a parametrisation of the family.
The pull-back $F^*\alpha\in\Lambda^1(\underline{L})$ vanishes on the tangent spaces
to the fibres of $\pi:\underline{L}\to B$ and therefore there is a unique fibrewise map
$$
\widetilde\alpha:\underline{L}\longrightarrow T^*B
$$
such that $\widetilde\alpha^*\lcan=F^*\alpha$.
Then $\widetilde{\mathcal L}$ is the image of the Legendrian embedding
$$
(F,-\widetilde \alpha): \underline{L} \longrightarrow X\times T^*B.
$$
Note that the Legendrian submanifold $\widetilde{\mathcal L}$
does not depend on the choice of the parametrisation of the family.
Furthermore, if $\alpha'=e^f\alpha$, then the associated Legendrian
submanifold $\widetilde{{\mathcal L}'}$ in $(X\times T^*B,\ker(\alpha'\oplus\lcan))$
is the image of $\widetilde{\mathcal L}$ under the contactomorphism~\eqref{StabInvar}.

\begin{exm}
\label{SuspConst}
The suspension of a constant family whose image is a Legendrian
submanifold $L\subset X$ is the product submanifold $L\times N^*B\subset X\times T^*B = X^B$,
where $N^*B$ is fancy notation for the zero section in $T^*B$
(the zero section is the conormal bundle of $B$ considered
as a submanifold of itself).
\end{exm}

\begin{exm}
\label{SuspTrans}
It is an immediate corollary of the definitions that a family
of Legendrian submanifolds in $X$ is transverse if and only if
its Legendrian suspension does not intersect the subset
$$
X\times N^*B \subset X\times T^*B = X^B,
$$
where again $N^*B$ is the zero section of $T^*B$.
\end{exm}

\begin{exm}[Suspension in $ST^*M$]
\label{SuspST*M}
In the case when $X=ST^*M$, it is natural to combine Examples~\ref{SuspConst}
and~\ref{SuspTrans} with the contactomorphism
$$
(ST^*M)^B \stackrel{\cong}{\longrightarrow} ST^*(M\times B) - \pi_B^*(ST^*B)
$$
from Example~\ref{StabST*M}:

\smallskip
\noindent
{\bf (i)}
The suspension of a constant family with image $L\subset ST^*M$ is
identified with the pull-back $\pi_M^*(L)\subset ST^*(M\times B)$,
where $\pi_M$ is the projection on~$M$.

\smallskip
\noindent
{\bf (ii)}
The image in $ST^*(M\times B)$ of the suspension of a transverse family
in $ST^*M$ is contained in the open set
$$
ST^*(M\times B) - \pi_B^*(ST^*B) - \pi_M^*(ST^*M),
$$
where $\pi_B$ and $\pi_M$ are the projections on $B$ and $M$,
respectively.
\end{exm}

\begin{rem}
Eliashberg and Polterovich~\cite[\S 2.2]{ElPo} introduced contact stabilisation
and Legendrian suspension in the case when $B$ is the circle and for families
of Legendrian submanifolds parametrised by product bundles $L\times S^1$,
i.e.\ for Legendrian loops with trivial monodromy.
Our Examples~\ref{SuspConst} and~\ref{SuspTrans} are straightforward
generalisations of the discussion of constant and positive Legendrian loops there.
\end{rem}

\subsection{Non-contractibility of transverse families}
Two families of Legendrian submanifolds in $(X,\ker\alpha)$ over the same base~$B$
are homotopic if they are restrictions to $B\times\{0\}$ and $B\times\{1\}$ of
a family of Legendrian submanifolds over~$B\times [0,1]$. A family is called
{\it contractible\/} if it is homotopic to a constant family.

\begin{thm}
\label{NoContrTrans}
Let $M$ be a closed manifold. Suppose that an isotopy class
of Legendrian submanifolds of $ST^*M$ contains $L_f$ for
a function $f:M\to\R$. Then there is no contractible
transverse family of Legendrian submanifolds over
a closed base $B$ in that class.
\end{thm}

\begin{proof}
The Legendrian suspensions of homotopic families over $B$ are Legendrian isotopic.
Using Example~\ref{StabST*M}, we identify the $B$-stabilisation of $ST^*M$
with $ST^*(M\times B)- \pi_B^*(ST^*B)$.
The suspension of a constant family with image~$L_f$
is the pull-back~$\pi_M^*(L_f)\subset ST^*(M\times B)$, whereas the suspension
of a transverse family lies in $ST^*(M\times B)-\pi_M^*(ST^*M)$,
see Example~\ref{SuspST*M}. Such Legendrian submanifolds cannot be Legendrian
isotopic by Theorem~\ref{NonDispl} applied to the product bundle $\pi_M:M\times B\to M$.
Hence, a transverse family in the Legendrian isotopy class of $L_f$
can not be contractible.
\end{proof}

As a first and typical example, let us apply this theorem to the Legendrian isotopy class
of the fibre of~$ST^*M$. Since a positive Legendrian loop is a transverse
family over the circle, we obtain the following result:

\begin{cor}
\label{PosLoops}
There are no contractible positive Legendrian loops in the Legendrian
isotopy class of the fibre of $ST^*M$.
\end{cor}

This corollary can only be meaningful for a closed manifold $M$
with finite fundamental group such that the integral cohomology
ring of its universal cover is generated by a single element,
because otherwise there are no positive loops in that Legendrian isotopy class
whatsoever by~\cite[Corollary~8.1]{ChNe2} and~\cite[Theorem~1.13]{FrLaSch}.

\begin{rem}[A shortcut to a partial order on $\widetilde{\mathrm{Cont}_0}(ST^*M)$]
\label{shortcut}
The nonexistence of contractible positive Legendrian loops in some
Legendrian isotopy class on a contact manifold~$X$ implies the nonexistence
of contractible positive loops of contactomorphisms in $\mathrm{Cont}_0(X)$.
Hence, combining Corollary~\ref{PosLoops}
with~\cite[Criterion 1.2.C]{ElPo}, we obtain a proof of the
universal orderability of ${\mathrm{Cont}_0}(ST^*M)$
for any closed manifold~$M$ that is independent of~\cite{ElKiPo}, \cite{ChNe2},
and~\cite{AlMe}.
\end{rem}

\section{Non-negative Legendrian isotopies and partial orders}

\begin{center}
{\it All Legendrian submanifolds are henceforth assumed closed and connected}.
\end{center}

\subsection{Non-negative Legendrian and contact isotopies}
\label{NonNegLeg&Cont}
A Legendrian isotopy $\{L_t\}_{t\in [0,1]}$ in a contact manifold $(X,\ker\alpha)$
is called non-negative if some (and hence every) parametrisation $\iota_t:L_0\to L_t$
satisfies
\begin{equation}
\label{DefLegNonneg}
\alpha\left(\tfrac{d}{dt}\iota_t(x)\right)\ge 0
\end{equation}
for all $t\in [0,1]$ and $x\in L_0$. Similarly, an isotopy of contactomorphisms
$\{\phi_t\}_{t\in [0,1]}$ is called non-negative if its contact Hamiltonian
\begin{equation}
\label{DefContNonneg}
H(\phi_t(x),t):=\alpha\left(\tfrac{d}{dt}\phi_t(x)\right)\ge 0
\end{equation}
for all $t\in [0,1]$ and~$x\in X$. If the inequalities in~\eqref{DefLegNonneg} and~\eqref{DefContNonneg}
are strict, the isotopies are called positive.

\begin{prop}
\label{ExtNonNeg}
If $\{L_t\}_{t\in [0,1]}$ is a non-negative Legendrian isotopy,
then there exists a compactly supported non-negative
contact isotopy $\{\phi_t\}_{t\in [0,1]}$
such that $\phi_0=\id_X$ and $\phi_t(L_0)=L_{\chi(t)}$
for a non-decreasing function $\chi:[0,1]\twoheadrightarrow [0,1]$.
\end{prop}

\begin{rem}
The Legendrian isotopy extension theorem asserts that for every parametrisation
$\iota_t:L_0\to L_t$ of an arbitrary Legendrian isotopy, there is a contact
isotopy $\phi_t$ such that $\phi_0=\id$ and ${\phi_t|}_{L_0}=\iota_t$,
see e.g.~\cite[Theorem~2.6.2]{Ge}.
However, it is not true that there exists a non-negative contact extension
for every parametrisation of a non-negative Legendrian isotopy. For instance,
a non-negative contact isotopy can only extend the constant parametrisation
of a constant isotopy, cf.~\cite[Proof of Lemma~4.12(i)]{GuiKaScha}.
\end{rem}

\begin{proof}[Proof of Proposition\/~{\rm \ref{ExtNonNeg}}]
By the Legendrian version of the Darboux--Weinstein theorem, a Legendrian submanifold $L$
has a neighbourhood $U$ contactomorphic to a neighbourhood~$U'$ of
the zero section in the $1$-jet bundle $\Jet^1(L)=\R\oplus T^*L$
with its canonical contact structure~$\ker(du-\lcan)$. Any Legendrian
submanifold sufficiently close to $L$ in $C^1$-topology is then
represented as the graph of the $1$-jet of a function $f:L\to\R$.
Non-negativity of a Legendrian isotopy of such graphs means simply
that, for the corresponding functions $f_t$ on $L$, the value $f_t(q)$
is a non-decreasing function of~$t$ for every~$q\in L$.
In this situation, it is easy to produce the required contact
extension supported in the neighbourhood~$U$. Indeed, suppose
that $f_t:L\to\R$ are functions such that $f_0\equiv 0$ and $\dot{f}_t\ge 0$.
The contact isotopy of $\Jet^1(L)$ given by
\begin{equation}
\label{JetShift}
\xi\longmapsto \xi + j^1 f_t(q)
\quad\mbox{ for } \xi\in \Jet^1_q(L), q\in L,
\end{equation}
is non-negative and extends the Legendrian isotopy of the graphs.
A compactly supported contact isotopy of $U\cong U'\subset\Jet^1(L)$
with the same properties is obtained by multiplying the contact Hamiltonian
of the  isotopy~\eqref{JetShift} by a non-negative cut-off function
$\psi: U'\to\R$ that is equal to~$1$ on a slightly smaller neighbourhood
of the zero section.

Now let us choose a subdivision $0=t_0<t_1<...<t_{k-1}<t_k=1$ of
the segment $[0,1]$ such that $L_t$, $t\in [t_j,t_{j+1}]$,
are suffciently close to $L_{t_j}$ in the sense of the preceding
paragraph. Then for each $j=0,...,k-1$ we obtain a compactly supported
non-negative contact isotopy $\{\phi_{j,t}\}_{t\in [t_j,t_{j+1}]}$
such that $\phi_{j,t_j}=\id_X$ and $\phi_{j,t}(L_{t_j})=L_{t}$.
The required isotopy $\{\phi_t\}_{t\in [0,1]}$ is a smoothened
concatenation of the isotopies $\{\phi_{j,t}\}$.
Namely, let $\chi:[0,1]\twoheadrightarrow [0,1]$ be any non-decreasing
smooth function such that
\begin{itemize}
\item[1)]
$\chi(t_j)=t_j$ for all $j=0,...,k$;
\item[2)]
$\{t\mid\chi'(t)=0\}=\{t\mid \chi^{(n)}(t)=0\mbox{ for all } n\in{\mathbb N}\}=\{t_1,...,t_{k-1}\}$.
\end{itemize}
Set $\widetilde{\phi}_{j,t}:=\phi_{j,\chi(t)}$
and define
$$
\phi_t:=
\widetilde{\phi}_{j,t}\circ\widetilde{\phi}_{j-1,t_{j}}\circ\cdots\circ\widetilde{\phi}_{0,t_1}
$$
for $t\in [t_j,t_{j+1}]$.
\end{proof}

\begin{rem}
\label{ChiDer}
Note for future use that the derivative of $\chi$ does not vanish at~$t_0=0$.
\end{rem}

\subsection{Non-negative and positive Legendrian loops}
A Legendrian loop is a family of Legendrian submanifolds over the circle.
It is often convenient to view loops as Legendrian isotopies $\{L_t\}_{t\in [0,1]}$
such that $L_0=L_1$ and the gluing at the endpoints is smooth. If the circle
is oriented, non-negative and positive Legendrian loops are defined
in the obvious way.

\begin{lem}
\label{SomewherePos}
If a non-negative Legendrian loop $\{L_\theta\}_{\theta\in S^1}$ is positive
on some non-empty interval~$I\subset S^1$, then it can be $C^\infty$-approximated
by a positive Legendrian loop.
\end{lem}

\begin{proof}
Let $[\theta_1,\theta_2]\subset I$ be an interval that is small enough
so that $L_\theta$, $\theta\in [\theta_1,\theta_2]$, can be represented
as graphs of $1$-jets of functions $f_\theta$ in $\Jet^1(L_{\theta_2})$.
The positivity of the loop on $I$ implies that this family of functions
is pointwise {\it strictly\/} increasing. In particular, the submanifolds
$L_{\theta_1}$ and $L_{\theta_2}$ are disjoint.

The restriction of our loop to $S^1-(\theta_1,\theta_2)$ is a non-negative
Legendrian isotopy with {\it disjoint\/} ends. By~\cite[Lemma~2.2]{ChNe2}
it can be $C^\infty$-approximated by a positive Legendrian isotopy
$\{\widetilde{L}_\theta\}_{\theta\in S^1-(\theta_1,\theta_2)}$ with the same ends.

It remains to interpolate between $f_{\theta_1}<0$ and $f_{\theta_2}\equiv 0$
by a new strictly increasing family of functions on $L_{\theta_2}$ so that
the union of the corresponding Legendrian isotopy over $[\theta_1,\theta_2]$
with $\{\widetilde L_\theta\}_{\theta\in S^1-(\theta_1,\theta_2)}$ is a smooth positive loop.
\end{proof}

\begin{prop}
\label{Nonneg2Pos}
If a Legendrian isotopy class contains a non-constant non-negative
{\rm (}contractible\/{\rm )} Legendrian loop,
then it contains a positive {\rm (}contractible\/{\rm )} Legendrian loop.
\end{prop}

\begin{proof}
A non-constant non-negative Legendrian loop in $(X,\ker\alpha)$ defines a non-negative
Legendrian isotopy $\iota_t:L_0\to L_t$, $t\in [0,1]$, such that $\iota_1(L_0)=L_0$ and
$$
\alpha\left({\tfrac{d}{dt}\iota_t|}_{t=0}(x_0)\right) >0
$$
for some $x_0\in L_0$. By Proposition~\ref{ExtNonNeg} and Remark~\ref{ChiDer},
there exists a compactly supported non-negative contact isotopy $\{\phi_t\}_{t\in[0,1]}$
such that $\phi_0\equiv\id_X$, $\phi_t(L_0)=L_{\chi(t)}$, and
$$
\alpha\left(
{\tfrac{d}{dt}\phi_t|}_{t=0}(x_0)\right)=
\chi'(0)\alpha\left({\tfrac{d}{dt}\iota_t|}_{t=0}(x_0)\right) >0.
$$
Let $U\subseteq L_0$ be a neighbourhood of $x_0$ on which this
strict inequality continues to hold.

Since $L_0$ is compact and connected, there exist contactomorphisms
$\Psi_j\in\mathrm{Cont}_0(X,\ker\alpha)$, $j=0,...,k$,
such that
\begin{itemize}
\item[1)]
$\Psi_0=\id_X$;
\item[2)]
$\Psi_j(L_0)=L_0$ for all $j$;
\item[3)]
$\Psi_j$ is isotopic to $\id_X$ within the class of contactomorphisms
preserving~$L_0$;
\item[4)]
$\bigcup\limits_{j=0}^{k} \Psi_j(U)=L_0$.
\end{itemize}
To construct $\Psi_j$, note first that there exist diffeomorphisms of $L_0$
isotopic to $\id_{L_0}$ and having property~(4). These diffeomorphisms
extend to contactomorphisms of $X$ with the required properties
by the Legendrian isotopy extension theorem.

Let $\psi_j:=(\Psi_j)^{-1}\circ \Psi_{j-1}$ for $j=1,...,k$ so that $(\Psi_j)^{-1}=\psi_j\circ\cdots\circ\psi_1$
for each $j=1,...,k$. Consider the contact isotopy
\begin{equation}
\label{EPformula}
\widetilde\phi_t:= \phi_t\circ\psi_k\circ\phi_t\circ\cdots\circ\psi_1\circ\phi_t,
\quad t\in[0,1].
\end{equation}
By property (4) of $\Psi_j$ and the choice of $U$, we have
$$
\alpha\left({\tfrac{d}{dt}\widetilde{\phi}_t|}_{t=0}(x)\right)>0
$$
for all $x\in L_0$. Hence, the Legendrian isotopy
$$
\widetilde L_t:= \widetilde\phi_t(L_0)
$$
is positive on some interval $[0,\epsilon)$, $\epsilon>0$.
Since $\widetilde L_0=\widetilde L_1$ by property (2),
smoothing $\{\widetilde L_t\}$ at $0$ and $1$ gives us
a non-negative Legendrian loop that is positive on
a slightly smaller interval $(\epsilon',\epsilon)$, $\epsilon>\epsilon'>0$.
By Lemma~\ref{SomewherePos}, this loop can be approximated
by an everywhere positive loop.

Finally, it follows from formula~\eqref{EPformula} and property (3) of $\Psi_j$ that
the obtained positive loop is homotopic to the $(k+1)$-fold iteration of the original
Legendrian loop. In particular, starting with a contractible non-constant
non-negative Legendrian loop, we get a contractible positive one.
\end{proof}

\begin{rem}
The proof of the proposition is a modification of the second
step in the proof of~\cite[Proposition~2.1.B]{ElPo}.
\end{rem}

\subsection{Partial orders on Legendrian isotopy classes and their universal coverings}
\label{PartOrders}
Let $L$ be a Legendrian submanifold in a contact manifold $(X,\ker\alpha)$.
Denote by $\Leg(L)$ the Legendrian isotopy class of $L$, i.e.\ the space of all
Legendrian submanifolds Legendrian isotopic to $L$ with $C^\infty$-topology,
and let $\Pi:\widetilde{\Leg}(L)\to\Leg(L)$ be the universal
covering of this space.

For $L_1,L_2\in\Leg(L)$, write $L_1\lle L_2$ if there is a non-negative
Legendrian isotopy connecting $L_1$ to~$L_2$. This partial relation admits
a natural lift to $\widetilde{\Leg}(L)$. For $\ell_1,\ell_2\in\widetilde{\Leg}(L)$,
write $\ell_1\ulle \ell_2$ if there exists a path $\gamma\subset \widetilde{\Leg}(L)$
connecting $\ell_1$ to~$\ell_2$ such that $\Pi(\gamma)$ is a non-negative Legendrian
isotopy.

It is clear that $\lle$ and $\ulle$ are reflexive and transitive. The Legendrian isotopy
class $\Leg(L)$ is called {\it orderable\/}
if $\lle$ is a partial order on it, i.e.\ if $\lle$ is also antisymmetric:
$$
L_1\lle L_2\mbox{ and } L_2\lle L_1 \Longrightarrow L_1=L_2.
$$
The class $\Leg(L)$ is called {\it universally orderable\/}
if $\ulle$ is a partial order on $\widetilde{\Leg}(L)$.

\begin{prop}
\label{Order}
A Legendrian isotopy class is orderable if and only if it does not
contain a positive Legendrian loop and universally orderable if
and only if it does not contain a contractible positive Legendrian
loop.
\end{prop}

\begin{proof}
The `only if' part is obvious from the definitions. It is equally
obvious that a class is (universally) orderable if it does not
contain a (contractible) non-constant non-negative Legendrian loop.
The result follows now from Proposition~\ref{Nonneg2Pos}.
\end{proof}

\begin{rem}
This proposition is a Legendrian version of~\cite[Criterion~1.2.C]{ElPo}.
\end{rem}

The following orderability result was obtained in~\cite{ChNe2}
in the case when $V$ is a point and is a consequence of~\cite[Corollary~4.14]{GuiKaScha}
and Proposition~\ref{ExtNonNeg} in the general case.

\begin{thm}
\label{OldOrder}
Let $M$ be a manifold with non-compact universal covering and $V\subset M$
a simply connected closed submanifold of codimension~$\ge 2$.
Then the Legendrian isotopy class $\Leg(SN^*V)$ of the spherical
conormal bundle of $V$ is orderable.
\end{thm}

However, $\Leg(SN^*V)$ is not orderable in general.
For instance, the condition that the universal covering of $M$
is non-compact is necessary if $\dim M\le 3$, see~\cite[Example~8.3]{ChNe2}.
The assumption that $V$ is simply connected can not be removed either,
as shown by the example of $V=S^1\times\{\mathrm{pt}\}\subset S^1\times S^2=M$.
Nevertheless, orderability can be restored by passing to the universal
covering of the Legendrian isotopy class.

\begin{thm}
\label{UnivOrder}
The Legendrian isotopy class $\Leg(SN^*V)$ of the spherical
conormal bundle of a connected closed submanifold $V\subset M$
of codimension~$\ge 2$ is {\bf\em universally} orderable.
\end{thm}

\begin{proof}
Suppose that $\Leg(SN^*V)$ is not universally orderable.
Then it contains a contractible positive Legendrian loop by Proposition~\ref{Order}.
Since this loop and its homotopy to a constant loop are compact,
we may assume that $M$ is a closed manifold. Recall now that $SN^*V$
is Legendrian isotopic to $L_f$ for a suitable function $f:M\to\R$,
see Example~\ref{Conormal}. Hence, the existence of
a contractible positive Legendrian loop (i.e.\ of a
contractible transverse family over the circle) in $\Leg(SN^*V)$
contradicts Theorem~\ref{NoContrTrans}.
\end{proof}

\begin{rem}
The argument in the proof of Theorem~\ref{UnivOrder} shows that
$\Leg(L_f)$ is universally orderable for every function $f:M\to \R$
on a closed manifold such that $0$ is not a critical value and
the hypersurface $\{f=0\}$ is connected.
\end{rem}

\end{document}